\documentclass[oneside,reqno,english]{amsart}
\usepackage[T1]{fontenc}
\usepackage[latin9]{inputenc}
\usepackage{babel}
\usepackage{mathrsfs}
\usepackage{enumitem}
\usepackage{amstext}
\usepackage{amsthm}
\usepackage{amssymb}
\PassOptionsToPackage{normalem}{ulem}
\usepackage{ulem}
\usepackage[unicode=true,pdfusetitle,
 bookmarks=true,bookmarksnumbered=false,bookmarksopen=false,
 breaklinks=false,pdfborder={0 0 1},backref=false,colorlinks=false]
 {hyperref}
\hypersetup{
 colorlinks=true,citecolor=blue,linkcolor=blue,linktocpage=true}

\makeatletter
\numberwithin{equation}{section}
\numberwithin{figure}{section}
\numberwithin{table}{section}
\theoremstyle{plain}
\newtheorem{thm}{\protect\theoremname}[section]
\theoremstyle{definition}
\newtheorem{defn}[thm]{\protect\definitionname}
\theoremstyle{plain}
\newtheorem{lem}[thm]{\protect\lemmaname}
\theoremstyle{plain}
\newtheorem{cor}[thm]{\protect\corollaryname}
\theoremstyle{remark}
\newtheorem{rem}[thm]{\protect\remarkname}
\theoremstyle{plain}
\newtheorem{prop}[thm]{\protect\propositionname}
\theoremstyle{definition}
\newtheorem{example}[thm]{\protect\examplename}
\theoremstyle{remark}
\newtheorem*{acknowledgement*}{\protect\acknowledgementname}

\allowdisplaybreaks
\usepackage{needspace}
\usepackage{bbm}

\setlist[enumerate]{itemsep=5pt,topsep=3pt}
\setlist[enumerate,1]{label=(\roman*),ref=\roman*}
\setlist[enumerate,2]{label=(\alph*),ref=\theenumi \alph*}

\AtBeginDocument{
  
}

\makeatother

\providecommand{\acknowledgementname}{Acknowledgement}
\providecommand{\corollaryname}{Corollary}
\providecommand{\definitionname}{Definition}
\providecommand{\examplename}{Example}
\providecommand{\lemmaname}{Lemma}
\providecommand{\propositionname}{Proposition}
\providecommand{\remarkname}{Remark}
\providecommand{\theoremname}{Theorem}

\begin{document}
\title[]{Realizations and factorizations of positive definite kernels}
\author{Palle Jorgensen and Feng Tian}
\address{(Palle E.T. Jorgensen) Department of Mathematics, The University of
Iowa, Iowa City, IA 52242-1419, U.S.A. }
\email{palle-jorgensen@uiowa.edu}
\urladdr{http://www.math.uiowa.edu/\textasciitilde jorgen/}
\address{(Feng Tian) Department of Mathematics, Hampton University, Hampton,
VA 23668, U.S.A.}
\email{feng.tian@hamptonu.edu}
\begin{abstract}
Given a fixed sigma-finite measure space $\left(X,\mathscr{B},\nu\right)$,
we shall study an associated family of positive definite kernels $K$.
Their factorizations will be studied with view to their role as covariance
kernels of a variety of stochastic processes. In the interesting cases,
the given measure $\nu$ is infinite, but sigma-finite. We introduce
such positive definite kernels $K\left(\cdot,\cdot\right)$ with the
two variables from the subset of the sigma-algebra $\mathscr{B}$,
sets having finite $\nu$ measure. Our setting and results are motivated
by applications. The latter are covered in the second half of the
paper. We first make precise the notions of realizations and factorizations
for $K$; and we give necessary and sufficient conditions for $K$
to have realizations and factorizations in $L^{2}\left(\nu\right)$.
Tools in the proofs rely on probability theory and on spectral theory
for unbounded operators in Hilbert space. Applications discussed here
include the study of reversible Markov processes, and realizations
of Gaussian fields, and their Ito-integrals.
\end{abstract}

\maketitle
\tableofcontents{}

\section{Introduction}

We study a family of problems from measurable dynamics and their connection
to the theory of positive definite functions (kernels). In particular,
the aim of the present paper is two-fold: (i) An extension of the
traditional reproducing kernel Hilbert space (RKHS) theory, from the
more traditional context of Aronszajn \cite{MR0008639,MR0051437}
to a measurable category. (ii) A characterization of when a positive
definite kernel assumes realization/factorization in the general setting.

Such an extension will adapt much better to a host of applications,
including problems from probability theory, from stochastic processes
\cite{AJ12,MR3507188}, from mathematical physics \cite{HKL14,RAKK05},
and measurable dynamics; and for the latter, especially to the context
of reversible processes (see, e.g., \cite{MR3039515,MR3290453,MR3046303,MR2982692,MR2811284,MR2815030},
and also \cite{MR3616046,MR3531697}). 

For applications to random processes, a kernel in the sense of Aronszajn
will typically represent a covariance kernel. Our applications include
a new spectral theoretic analysis of (a) transient Markov processes,
and of (b) generalized Gaussian fields and their Ito-integrals.

Now in the standard approach to RKHSs of Aronszajn, one starts with
a positive definite (p.d.) function, $K$ on $X\times X$ (often called
a p.d. kernel), where $X$ is a given set. The term \textquotedblleft reproducing\textquotedblright{}
refers to the fact that for every $f$ in $\mathscr{H}$, the values
$f\left(x\right)$ can be reproduced from the inner Hilbert-product
in $\mathscr{H}$. With a standard construction, starting with $X$
and $K$, one then arrives at a Hilbert space $\mathscr{H}$ of functions
on $X$, the so called reproducing kernel Hilbert space (RKHS). It
depends on the pair $\left(X,K\right)$ of course; so is denoted $\mathscr{H}\left(K\right)$
when the kernel is not given from the context. \emph{A priori}, the
set $X$ is not given any additional structure, but a key point is
that both $K$ and the functions $f$ in the RKHS $\mathscr{H}\left(K\right)$
are defined everywhere on $X$. If for example, $X$ is a complex
domain, in interesting applications, then the functions in $\mathscr{H}\left(K\right)$
will be analytic, or in the case of the familiar RKHS of Bargmann,
the functions in $\mathscr{H}\left(K\right)$ will be entire analytic.
If $X$ has a topology, and if $K$ is assumed continuous, then the
functions in $\mathscr{H}\left(K\right)$ will then also be continuous.

A novelty in the present paper is the starting point being a fixed
measure space $\left(X,\mathscr{B},\nu\right)$ where $\nu$ is assumed
sigma-finite. Set $\mathscr{B}_{fin}:=\left\{ A\in\mathscr{B}\mathrel{;}\nu\left(A\right)<\infty\right\} $.
We shall then consider p.d. functions $K$ on $\mathscr{B}_{fin}\times\mathscr{B}_{fin}$.
Our main result is a characterization of the case when the corresponding
RKHS $\mathscr{H}\left(K\right)$ may be realized in $L^{2}\left(\nu\right)$. 

In addition to a number of applications, we also give necessary and
sufficient conditions on $K$ which yield factorizations
\[
K\left(A,B\right)=\int_{X}k_{A}\left(x\right)\overline{k_{B}\left(x\right)}d\nu\left(x\right),\quad\forall A,B\in\mathscr{B}_{fin},
\]
where $\left\{ k_{A}\right\} _{A\in\mathscr{B}_{fin}}$ is an $L^{2}\left(\nu\right)$
system. 

But up to now, many of the applications have focused on Hilbert spaces
of regular functions. If for example, a kernel represents a Green's
function for an elliptic partial differential operator (PDO), then
the associated RKHS will consist of functions which have some degree
of smoothness.

The Aronszajn approach has serious limitations: Often functions will
be defined only almost everywhere with respect to some measure which
is prescribed on the set $X$; for example, if $X$ represents time,
in one or more dimensions, the prescribed measure $\nu$ is often
Lebesgue measure. For fractal random fields, $\nu$ may be a fractal
measure. For this reason, and others (to be outlined inside the paper),
it is useful to instead let $X$ be a measure space, say $\left(X,\mathscr{B},\nu\right)$.
If $X$ is a locally compact Hausdorff space, then $\mathscr{B}$
is the corresponding sigma-algebra of Borel sets, and $\nu$ is a
fixed positive measure, and assumed to be a regular measure on $\left(X,\mathscr{B}\right)$.
The modification in the resulting new definition of p.d. kernels $K$
in this context is subtle. Here we just mention that, for a p.d. system
$\left(X,\mathscr{B},\nu\right)$ and kernel $K$ in the measurable
category, the associated RKHS $\mathscr{H}\left(K\right)$ will now
instead be a Hilbert space of measurable functions on $X$, more precisely,
measurable with respect to $\mathscr{B}$, and locally in $L^{2}\left(\nu\right)$.
We shall say that $\mathscr{H}\left(K\right)$ is contained in $L_{loc}^{2}\left(\nu\right)$.
The p.d. kernel $K$ itself will be a random family of signed measures
on $\left(X,\mathscr{B}\right)$.

\textbf{Discussion of the literature.} The theory of RKHS and their
applications is vast, and below we only make a selection. Readers
will be able to find more cited there. As for the general theory of
RKHS in the pointwise category, we find useful \cite{AD92,ABDdS93,AD93,MR2529882,MR3526117}.
The applications include fractals (see e.g., \cite{MR0008639,AJSV13});
probability theory \cite{MR0277027,MR3504608,MR3624688,MR3571702};
and application to learning theory \cite{MR2058288,MR2558684,MR3450534}.
For recent applications, we refer to \cite{MR3450534,MR3441738,MR3559001}. 

\section{\label{sec:fact}Factorization}

We study representations of positive definite kernels $K$ in a general
setting, but with view to applications to realizations of certain
stochastic processes, to harmonic analysis, and to metric geometry.
Our initial results are stated for the most general given positive
definite kernel, but are then subsequently specialized to the above
mentioned applications. Given a positive definite kernel $K$ on $X\times X$
where $X$ is a fixed set, we first study families of \emph{factorizations}
(or representations) of $K$. 
\begin{defn}
If $K:X\times X\longrightarrow\mathbb{C}$ is a given positive definite
kernel, and $\left(\Omega,\mathscr{M},\mu\right)$ is a measure space,
a function $k:X\longrightarrow L^{2}\left(\Omega,\mathscr{M},\mu\right)$
is said to be a factorization iff (Defn.) 
\[
k\left(x,y\right)=\left\langle k_{x},k_{y}\right\rangle _{L^{2}\left(\Omega,\mathscr{M},\mu\right)},\quad\forall x,y\in X.
\]
\end{defn}

Let $\left(X,\mathscr{B},\nu\right)$ denote a measure space. Our
starting assumptions are as follows:
\begin{enumerate}
\item \label{enu:a1}$\mathscr{B}$ is a fixed sigma-algebra of subsets
of $X$, and 
\item \label{enu:a2}$\nu$ is a positive $\sigma$-finite measure defined
on $\mathscr{B}$. 
\end{enumerate}
Given (\ref{enu:a1}) \& (\ref{enu:a2}), we set $\mathscr{B}_{fin}=\left\{ A\in\mathscr{B}\mathrel{;}\nu\left(A\right)<\infty\right\} $. 

We shall consider positive definite kernels $K$ on $\mathscr{B}_{fin}\times B_{fin}$.
While our results are valid in the case of function spaces consisting
of complex-valued functions, we shall restrict attention here to the
real-valued case. This is motivated by applications, and the extension
to the complex case is relatively simple. 
\begin{defn}
\label{def:pdk}A function $K$ on $\mathscr{B}_{fin}\times B_{fin}$
(mapping into $\mathbb{R}$) is said to be \emph{positive definite}
iff (Defn.), for $\forall n\in\mathbb{N}$, $\forall\left\{ \alpha_{i}\right\} _{1}^{n}$,
$\alpha_{i}\in\mathbb{R}$, and all $\left\{ A_{i}\right\} _{1}^{n}$,
$A_{i}\in\mathscr{B}_{fin}$, we have 
\begin{equation}
\sum\nolimits _{i}\sum\nolimits _{j}\alpha_{i}\alpha_{j}K\left(A_{i},A_{j}\right)\geq0.\label{eq:m1}
\end{equation}

When a positive definite kernel $K$ is given, we shall denote the
corresponding reproducing kernel Hilbert space (RKHS) by $\mathscr{H}\left(K\right)$. 
\end{defn}

Recall the RKHS $\mathscr{H}\left(K\right)$ is a Hilbert space of
functions $F$ defined on $\mathscr{B}_{fin}$ such that, for all
$A\in\mathscr{B}_{fin}$, the mapping 
\begin{equation}
\mathscr{H}\left(K\right)\ni F\longrightarrow F\left(A\right)\label{eq:m2}
\end{equation}
is norm-continuous. When this holds, there is a unique representation
$\varphi_{A}\in\mathscr{H}\left(K\right)$ such that 
\begin{equation}
F\left(A\right)=\left\langle F,\varphi_{A}\right\rangle _{\mathscr{H}\left(K\right)}.\label{eq:m3}
\end{equation}
It is well-known that one may take $\varphi_{A}\left(\cdot\right)=K\left(\cdot,A\right)$,
and that then 
\begin{equation}
\left\langle K\left(\cdot,A\right),K\left(\cdot,B\right)\right\rangle _{\mathscr{H}\left(K\right)}=K\left(A,B\right),\quad\forall A,B\in\mathscr{B}_{fin}.\label{eq:m4}
\end{equation}

\begin{defn}
\label{def:pdf}We shall say that a given positive definite kernel
$K$ on $\mathscr{B}_{fin}\times\mathscr{B}_{fin}$ has a \emph{realization}
(as factorization) in $L^{2}\left(\nu\right)$ iff (Defn.) $\exists$$\left\{ k_{A}\right\} _{A\in\mathscr{B}_{fin}}$,
such that $k_{A}\in L^{2}\left(\nu\right)$, and 
\begin{equation}
K\left(A,B\right)=\left\langle k_{A},k_{B}\right\rangle _{L^{2}\left(\nu\right)}=\int_{X}k_{A}\left(x\right)k_{B}\left(x\right)d\nu\left(x\right).\label{eq:m5}
\end{equation}
\end{defn}

\begin{thm}
\label{thm:f1}Let $X$, $\mathscr{B}$, $\nu$, and $K$ be as above,
i.e., we assume that $K$ is a positive definite kernel defined on
$\mathscr{B}_{fin}\times\mathscr{B}_{fin}$. 

Then the following two conditions are equivalent: 
\begin{enumerate}
\item \label{enu:at1}For all $B\in\mathscr{B}_{fin}$, $K\left(\cdot,B\right)$
is a sigma-finite signed measure on $\mathscr{B}$, and 
\begin{equation}
K\left(\cdot,B\right)\ll\nu\label{eq:m6}
\end{equation}
where $\ll$ refers to ``absolute continuity,'' i.e., that the following
implication holds: 
\[
\nu\left(A\right)=0\Longrightarrow K\left(A,B\right)=0.
\]
\item \label{enu:at2}The positive definite kernel $K$ in (\ref{enu:at1})
has an $L^{2}\left(\nu\right)$-realization, i.e., (\ref{eq:m5})
is satisfied for a system $\left\{ k_{A}\right\} _{A\in\mathscr{B}_{fin}}$. 
\end{enumerate}
\end{thm}

Before we turn to the proof, we shall need the following lemma. It
clarifies the notation of condition (\ref{eq:m6}) in (\ref{enu:at1}).
\begin{lem}
Let $\left(X,\mathscr{B},\nu\right)$, $K$ and $\mathscr{B}_{fin}$
be as stated in the theorem. Let $A\in\mathscr{B}_{fin}$ be given.
If $\nu\left(A\right)=0$ $\Longrightarrow$ $K\left(A,A\right)=0$,
then it follows automatically that also $K\left(A,B\right)=0$ for
all $B\in\mathscr{B}_{fin}$.
\end{lem}

\begin{proof}
The assertion follows from the assumed positive definite property
for $K$. If $\mathscr{H}\left(K\right)$ denotes the RKHS, then the
Schwarz inequality holds for the $\mathscr{H}\left(K\right)$-inner
product. Hence 
\begin{eqnarray*}
\left|K\left(A,B\right)\right|^{2} & = & \left|\left\langle K\left(\cdot,A\right),K\left(\cdot,B\right)\right\rangle _{\mathscr{H}\left(K\right)}\right|^{2}\\
 & \underset{\text{Schwarz}}{\leq} & \left\Vert K\left(\cdot,A\right)\right\Vert _{\mathscr{H}\left(K\right)}^{2}\left\Vert K\left(\cdot,B\right)\right\Vert _{\mathscr{H}\left(K\right)}^{2}\\
 & \underset{\text{by }\left(\ref{eq:m3}\right)\&\left(\ref{eq:m4}\right)}{=} & K\left(A,A\right)K\left(B,B\right);
\end{eqnarray*}
and the conclusion of the lemma follows.
\end{proof}
\begin{proof}[Proof of Theorem \ref{thm:f1}]
(\ref{enu:at1})$\Rightarrow$(\ref{enu:at2}). With (\ref{enu:at1})
assumed, we shall denote the indexed family of Radon-Nikodym derivatives
$g\left(\cdot,B\right)$; so an indexed family of locally integrable
functions on $\left(X,\mathscr{B}\right)$ such that 
\begin{equation}
\frac{dK\left(\cdot,B\right)}{d\nu}=g\left(\cdot,B\right);\label{eq:m8}
\end{equation}
or equivalently, 
\begin{equation}
\int_{A}g\left(x,B\right)d\nu\left(x\right)=K\left(A,B\right),\quad\forall A,B\in\mathscr{B}_{fin}.\label{eq:m9}
\end{equation}

Let $\mathscr{D}_{fin}$ denote the finite linear combinations 
\begin{equation}
\sum\nolimits _{i=1}^{n}\alpha_{i}\chi_{A_{i}},\quad\alpha_{i}\in\mathbb{R},\;A_{i}\in\mathscr{B}_{fin}.\label{eq:m10}
\end{equation}
Recall that $\mathscr{D}_{fin}$ is dense in $L^{2}\left(\nu\right)$
relative to the $L^{2}\left(\nu\right)$-norm. For $\varphi=\sum_{i}\alpha_{i}\chi_{A_{i}}$,
$\psi=\sum_{j}\beta_{j}\chi_{B_{j}}$ $\in\mathscr{D}_{fin}$, $\alpha_{i},\beta_{j}\in\mathbb{R}$,
$A_{i},B_{j}\in\mathscr{B}_{fin}$; then 
\begin{equation}
\varphi,\psi\longmapsto\sum\nolimits _{i}\sum\nolimits _{j}\alpha_{i}\beta_{j}\int_{A_{i}}g\left(x,B_{j}\right)d\nu\left(x\right)\label{eq:m11}
\end{equation}
defines a positive definite bilinear mapping in $L^{2}\left(\nu\right)$,
with $\mathscr{D}_{fin}$ as its dense domain. One easily checks that
it is closable. Hence an application of the Kato-Friedrichs theorem
(see, e.g., \cite{MR1009163,MR1335452,MR3642406}) yields a positive
selfadjoint operator $T$ in $L^{2}\left(\nu\right)$, $\mathscr{D}_{fin}\subset dom\left(T\right)$
such that the expression in (\ref{eq:m11}) takes the following form
\begin{align}
\text{RHS}_{\left(\ref{eq:m11}\right)} & =\int_{X}\varphi\left(x\right)\left(T\psi\right)\left(x\right)d\nu\left(x\right)\label{eq:m12}\\
 & =\left\langle \varphi,T\psi\right\rangle _{L^{2}\left(\nu\right)}=\langle T^{\frac{1}{2}}\varphi,T^{\frac{1}{2}}\psi\rangle_{L^{2}\left(\nu\right)}\nonumber \\
 & =\int_{X}(T^{\frac{1}{2}}\varphi)\left(x\right)(T^{\frac{1}{2}}\psi)\left(x\right)d\nu\left(x\right).\nonumber 
\end{align}
Setting $\varphi\in\chi_{A}$, $\psi=\chi_{B}$, $A,B\in\mathscr{B}_{fin}$,
and combining (\ref{eq:m11}) \& (\ref{eq:m12}), we get 
\[
K\left(A,B\right)=\int_{X}(T^{\frac{1}{2}}\varphi)\left(x\right)(T^{\frac{1}{2}}\psi)\left(x\right)d\nu\left(x\right).
\]
This yields a solution to (\ref{enu:at2}) with $k_{A}\left(x\right)=(T^{\frac{1}{2}}(\chi_{A}))\left(x\right)$,
$A\in\mathscr{B}_{fin}$, $x\in X$. 

We now turn to (\ref{enu:at2})$\Rightarrow$(\ref{enu:at1}). Assuming
$\left\{ k_{A}\right\} _{A\in\mathscr{B}_{fin}}$ is a solution to
(\ref{enu:at2}), see also (\ref{eq:m5}), we may define an operator
$S$ with dense domain $\mathscr{D}_{fin}$ in $L^{2}\left(\nu\right)$,
by setting 
\begin{equation}
S\left(\sum\nolimits _{i}\alpha_{i}\chi_{A_{i}}\right)=\sum\nolimits _{i}\alpha_{i}k_{A_{i}}.\label{eq:m13}
\end{equation}

From the condition in (\ref{enu:at2}) we conclude that $S$ is closable,
and that $S^{*}S$ is defined on $\mathscr{D}_{fin}$. Hence for $\varphi,\psi\in\mathscr{D}_{fin}$;
see (\ref{eq:m10})-(\ref{eq:m11}), we get 
\[
\sum\nolimits _{i}\sum\nolimits _{j}\alpha_{i}\beta_{j}K\left(A_{i},B_{j}\right)=\int_{X}\varphi\left(x\right)\left(S^{*}S\psi\right)\left(x\right)d\nu\left(x\right).
\]
The conclusion in (\ref{enu:at1}) is now immediate: 

For the associated Radon-Nikodym derivative (see (\ref{eq:m8})) we
get 
\[
\frac{dK\left(x,B\right)}{d\nu\left(x\right)}=g\left(x,B\right)=\left(S^{*}S\left(\chi_{B}\right)\right)\left(x\right).
\]
\end{proof}
\begin{cor}
\label{cor:f1}Let $\left(X,\mathscr{B},\nu\right)$ be a fixed sigma-finite
measure-space, as specified in Theorem \ref{thm:f1} above. We set
$\mathscr{B}_{fin}:=\left\{ A\in\mathscr{B}\mathrel{;}\nu\left(A\right)<\infty\right\} $,
and we consider a positive kernel $K$ defined on $\mathscr{B}_{fin}\times\mathscr{B}_{fin}$.
Suppose $K$ has a factorization $\left\{ k_{A}\right\} _{A\in\mathscr{B}_{fin}}$
in $L^{2}\left(\nu\right)$, i.e., 
\begin{equation}
K\left(A,B\right)=\left\langle k_{A},k_{B}\right\rangle _{L^{2}\left(\nu\right)}\label{eq:n1}
\end{equation}
where the RHS in (\ref{eq:n1}) refers to the $L^{2}\left(\nu\right)$-inner
product. 

Define a function $b$ by
\[
b\left(K\left(\cdot,A\right)\right)=k_{A}
\]
from $\mathscr{H}\left(K\right)(=$ the RKHS$)$ into $L^{2}\left(\nu\right)$,
and then extend by linearity: 
\begin{equation}
b\left(\sum\nolimits _{i}\alpha_{i}K\left(\cdot,A_{i}\right)\right)=\sum_{i}\alpha_{i}k_{A_{i}};\label{eq:n2}
\end{equation}
then $b$ is an \uline{isometry} $\mathscr{H}\left(K\right)\longrightarrow L^{2}\left(\nu\right)$. 
\begin{enumerate}
\item \label{enu:n1}When $b$ is extended by closure, we get $b\left(\mathscr{H}\left(K\right)\right)\subseteq L^{2}\left(\nu\right)$
as a closed subspace.
\item \label{enu:n2}The co-isometry $L^{2}\left(\nu\right)\xrightarrow{\;b^{*}\;}\mathscr{H}\left(K\right)$
is given as follows: For all $\varphi\in L^{2}\left(\nu\right)$,
and $A\in\mathscr{B}_{fin}$, set 
\begin{equation}
\left(b^{*}\varphi\right)\left(A\right)=\left\langle \varphi,k_{A}\right\rangle _{L^{2}\left(\nu\right)}=\int_{X}\varphi\left(x\right)k_{A}\left(x\right)d\nu\left(x\right),\label{eq:o1}
\end{equation}
then $b^{*}$ is the adjoint operator to $b$ from (\ref{enu:n1}). 
\item \label{enu:n3}For every orthonormal basis (ONB) in $L^{2}\left(\nu\right)$,
$\left\{ \varphi_{n}\right\} _{n\in\mathbb{N}}$, we have the corresponding
factorization of $K\left(\cdot,\cdot\right)$ on $\mathscr{B}_{fin}\times\mathscr{B}_{fin}$:
\begin{equation}
K\left(A,B\right)=\sum\nolimits _{n\in\mathbb{N}}\left\langle \varphi_{n},k_{A}\right\rangle _{L^{2}\left(\nu\right)}\left\langle k_{B},\varphi_{n}\right\rangle _{L^{2}\left(\nu\right)}.\label{eq:o2}
\end{equation}
 
\end{enumerate}
\end{cor}

\begin{proof}
The key step is the isometry-property via formula (\ref{eq:n2}),
i.e., the fact that 
\begin{equation}
\left\Vert \sum\nolimits _{i}\alpha_{i}K\left(\cdot,A_{i}\right)\right\Vert _{\mathscr{H}\left(K\right)}^{2}=\left\Vert \sum\nolimits _{i}\alpha_{i}k_{A_{i}}\right\Vert _{L^{2}\left(\nu\right)}^{2}.\label{eq:n3}
\end{equation}
To prove this, fix $\left\{ \alpha_{i}\right\} _{1}^{n}$, $\left\{ A_{i}\right\} _{1}^{n}$
as specified; then 
\begin{eqnarray*}
\text{LHS}_{\left(\ref{eq:n3}\right)} & = & \sum\nolimits _{i}\sum\nolimits _{j}\alpha_{i}\alpha_{j}\left\langle K\left(\cdot,A_{i}\right),K\left(\cdot,A_{j}\right)\right\rangle _{\mathscr{H}\left(K\right)}\\
 & = & \sum\nolimits _{i}\sum\nolimits _{j}\alpha_{i}\alpha_{j}K\left(A_{i},A_{j}\right)\quad(\text{the reproducing property})\\
 & \underset{\text{by \ensuremath{\left(\ref{eq:n1}\right)}}}{=} & \sum\nolimits _{i}\sum\nolimits _{j}\alpha_{i}\alpha_{j}\int_{X}k_{A_{i}}\left(x\right)k_{A_{j}}\left(x\right)d\nu\left(x\right)\\
 & = & \int_{X}\left|\sum\nolimits _{i}\alpha_{i}k_{A_{i}}\left(x\right)\right|^{2}d\nu\left(x\right)=\text{RHS}_{\left(\ref{eq:n3}\right)},
\end{eqnarray*}
the desired conclusion.

Hence $b$ extends by closure to an isometry, also denoted by $b$,
$\mathscr{H}\left(K\right)\longrightarrow L^{2}\left(\nu\right)$,
and we conclude that $b\left(\mathscr{H}\left(K\right)\right)$ is
closed in $L^{2}\left(\nu\right)$.

\emph{Proof of (\ref{enu:n2}).} To begin with we first establish
that $b^{*}$ from (\ref{eq:o1}) yields a well defined operator from
$L^{2}\left(\nu\right)$ into $\mathscr{H}\left(K\right)$. Note that
this makes reference to the respective inner products on the two Hilbert
spaces. In particular for $\mathscr{H}\left(K\right)$, a finite linear
combination (see Definition \ref{def:pdk}), 
\begin{equation}
\left\Vert \sum\nolimits _{i}\alpha_{i}K\left(\cdot,A_{i}\right)\right\Vert _{\mathscr{H}\left(K\right)}^{2}=\sum\nolimits _{i}\sum\nolimits _{j}\alpha_{i}\alpha_{j}K\left(A_{i},A_{j}\right),\label{eq:o3}
\end{equation}
with $\alpha_{i}\in\mathbb{R}$, and $A_{i}\in\mathscr{B}_{fin}$. 

To show that (\ref{eq:o1}) defines an operator as stated into $\mathscr{H}\left(K\right)$,
we must show that $b^{*}\varphi\in\mathscr{H}\left(K\right)$, and
\begin{equation}
\left|\sum\nolimits _{i}\alpha_{i}\left(b^{*}\varphi\right)\left(A_{i}\right)\right|^{2}\leq\left\Vert \varphi\right\Vert _{L^{2}\left(\nu\right)}^{2}\sum\nolimits _{i}\sum\nolimits _{j}\alpha_{i}\alpha_{j}K\left(A_{i},A_{j}\right).\label{eq:o4}
\end{equation}
We have 
\begin{eqnarray*}
\text{LHS}_{\left(\ref{eq:o4}\right)} & \underset{\left(\text{by }\left(\ref{eq:o1}\right)\right)}{=} & \left|\left\langle \varphi,\sum\nolimits _{i}\alpha_{i}k_{A_{i}}\right\rangle _{L^{2}\left(\nu\right)}\right|^{2}\\
 & \underset{\left(\text{by Schwarz}\right)}{\leq} & \left\Vert \varphi\right\Vert _{L^{2}\left(\nu\right)}^{2}\left\Vert \sum\nolimits _{i}\alpha_{i}k_{A_{i}}\right\Vert _{L^{2}\left(\nu\right)}^{2}.
\end{eqnarray*}
Hence $b^{*}\varphi\in\mathscr{H}\left(K\right)$ by Aronszajn \cite{MR0051437},
and (\ref{eq:n3}). 

Indeed, by (\ref{enu:n1}), we also have 
\begin{eqnarray*}
\left\Vert \sum\nolimits _{i}\alpha_{i}k_{A_{i}}\right\Vert _{L^{2}\left(\nu\right)}^{2} & = & \sum\nolimits _{i}\sum\nolimits _{j}\alpha_{i}\alpha_{j}\left\langle k_{A_{i}},k_{A_{j}}\right\rangle _{L^{2}\left(\nu\right)}\\
 & \underset{\left(\text{Thm. \ref{thm:f1}}\right)}{=} & \sum\nolimits _{i}\sum\nolimits _{j}\alpha_{i}\alpha_{j}K\left(A_{i},A_{j}\right),
\end{eqnarray*}
and the result follows; see also (\ref{eq:o3}). 

To show that $b^{*}$ as an operator is the adjoint, we must check
that 
\[
\left\langle b^{*}\varphi,F\right\rangle _{\mathscr{H}\left(K\right)}=\left\langle \varphi,bF\right\rangle _{L^{2}\left(\nu\right)},\quad\forall\varphi\in L^{2}\left(\nu\right),\:\forall F\in\mathscr{H}\left(K\right).
\]
But by \cite{MR0051437}, it is enough to do this for the special
case when $F=K\left(\cdot,A\right)$, $A\in\mathscr{B}_{fin}$; and
in this case, it follows from (\ref{enu:n1}).

\emph{Proof of (\ref{enu:n3}).} Let $\left\{ \varphi_{n}\right\} $
be an ONB in $L^{2}\left(\nu\right)$. By the assumptions, and Theorem
\ref{thm:f1}, we get 
\begin{eqnarray*}
K\left(A,B\right) & = & \left\langle k_{A},k_{B}\right\rangle _{L^{2}\left(\nu\right)}\\
 & \underset{\left(\text{Parseval}\right)}{=} & \sum\nolimits _{n\in\mathbb{N}}\left\langle k_{A},\varphi_{n}\right\rangle _{L^{2}\left(\nu\right)}\left\langle \varphi_{n},k_{B}\right\rangle _{L^{2}\left(\nu\right)}
\end{eqnarray*}
which is the desired conclusion. 
\end{proof}
\begin{rem}
It follows from the argument of the proof that $b\left(\mathscr{H}\left(K\right)\right)=L^{2}\left(\nu\right)$
\emph{iff} $\left\{ k_{A}\right\} _{A\in\mathscr{B}_{fin}}^{\perp}=0$,
where $\perp$ refers to the $L^{2}\left(\nu\right)$-inner product.
See Example \ref{exa:1d} for case when $b\left(\mathscr{H}\left(K\right)\right)$
is one-dimensional.
\end{rem}

\section{\label{sec:MF}Minimal Factorization}

Below we show that any solution $\left\{ k_{A}\right\} _{A\in\mathscr{B}_{fin}}$
to the factorization problem from Theorem \ref{thm:f1} must necessarily
be minimal in a sense we defined below; see also \cite{2017arXiv170609532J}.
\begin{defn}
Let $\left(X,\mathscr{B},\nu\right)$, and $K$ be as in the statement
of Theorem \ref{thm:f1}. Let $\left(\Omega,\mathscr{M},\mu\right)$
be a sigma-finite measure space. Suppose we have two factorizations
\begin{align}
\mathscr{B}_{fin} & \ni A\longmapsto k_{A}\in L^{2}\left(X,\nu\right),\quad\text{and}\label{eq:p1}\\
\mathscr{B}_{fin} & \ni A\longmapsto m_{A}\in L^{2}\left(\Omega,\mu\right),\label{eq:p2}
\end{align}
such that 
\begin{align*}
K\left(A,B\right) & =\int_{X}k_{A}\left(x\right)k_{B}\left(x\right)d\nu\left(x\right)\\
 & =\int_{\Omega}m_{A}\left(\omega\right)m_{B}\left(\omega\right)d\mu\left(\omega\right),\quad\forall A,B\in\mathscr{B}_{fin}.
\end{align*}
We say that $k\ll m$ if (Defn.) there is a measurable mapping $F:\left(\Omega,\mathscr{M}\right)\longrightarrow\left(X,\mathscr{B}\right)$
such that 
\begin{equation}
\mu\circ F^{-1}=\nu.\label{eq:p3}
\end{equation}
\end{defn}

\begin{thm}
Solution $\left\{ k_{A}\right\} _{A\in\mathscr{B}_{fin}}$ to the
factorization problem (\ref{eq:m5}), if they exist, are minimal in
the sense that whenever some other solution $\left\{ k'_{A}\right\} _{A\in\mathscr{B}_{fin}}$
in a sigma-finite measure space $\left(X',\mathscr{B}',\nu'\right)$
satisfies $k'\ll k$, then the two measure spaces $\left(X,\mathscr{B},\nu\right)$
and $\left(X',\mathscr{B}',\nu'\right)$ are isomorphic. 
\end{thm}

\begin{proof}
We refer the reader to our \cite{2017arXiv170609532J,MR3642406}. 
\end{proof}

\section{Two applications}

The positive definite kernels, and their factorizations, considered
in Sections \ref{sec:fact}-\ref{sec:MF}, include as special cases
covariance kernels of general white noise processes, of transient
Markov chains, and of generalized Gaussian fields. Below we discuss
applications of our factorization results to these three cases. Our
treatment of generalized Gaussian fields includes an extension of
the more traditional setting for Ito calculus. This is in the last
section of the paper.

\subsection{The Generalized Wiener-process}

Here we consider the following kernel $K$ on $\mathscr{B}_{fin}\times\mathscr{B}_{fin}$:
Set 
\begin{equation}
K\left(A,B\right)=\nu\left(A\cap B\right),\quad A,B\in\mathscr{B}_{fin},\label{eq:n4}
\end{equation}
where the measure space $\left(X,\mathscr{B},\nu\right)$ is specified
as above. 
\begin{prop}
~
\begin{enumerate}
\item $K=K^{\left(\nu\right)}$ in (\ref{eq:n4}) is positive definite.
\item $K^{\left(\nu\right)}$ is the covariance kernel for the stationary
Wiener process $W=W^{\left(\nu\right)}$ indexed by $\mathscr{B}_{fin}$,
i.e., Gaussian, mean zero, and 
\begin{equation}
\mathbb{E}\left(W_{A}W_{B}\right)=K^{\left(\nu\right)}\left(A,B\right)=\nu\left(A\cap B\right).\label{eq:n5}
\end{equation}
\item If $f\in L^{2}\left(\nu\right)$, and $W_{f}=\int_{X}f\left(x\right)dW_{x}$
denotes the corresponding Ito-integral, then 
\[
\mathbb{E}\left(\left|W_{f}\right|^{2}\right)=\int_{X}\left|f\right|^{2}d\nu;
\]
in particular,if $f=\sum_{i}\alpha_{i}\chi_{A_{i}}$, then 
\[
\sum\nolimits _{i}\sum\nolimits _{j}\alpha_{i}\alpha_{j}K^{\left(\nu\right)}\left(A_{i},A_{j}\right)=\int_{X}\left|\sum\nolimits _{i}\alpha_{i}\chi_{A_{i}}\right|^{2}d\nu.
\]
\item The RKHS $\mathscr{H}\left(K^{\left(\nu\right)}\right)$ of the positive
definite kernel in (\ref{eq:n4}) consists of functions $F$ on $\mathscr{B}_{fin}$
represented by $f\in L^{2}\left(\nu\right)$ via 
\begin{equation}
F\left(A\right)=F_{f}\left(A\right)=\int_{A}fd\nu,\quad A\in\mathscr{B}_{fin};\label{eq:n5b}
\end{equation}
and 
\begin{equation}
\left\Vert F_{f}\right\Vert _{\mathscr{H}\left(K\right)}^{2}=\left\Vert f\right\Vert _{L^{2}\left(\nu\right)}^{2}=\int_{X}\left|f\right|^{2}d\nu.\label{eq:n6}
\end{equation}
\item The isometry $b:\mathscr{H}\left(K\right)\longrightarrow L^{2}\left(\nu\right)$
from Corollary \ref{cor:f1}, is specified by 
\begin{equation}
b\left(K^{\left(\nu\right)}\left(\cdot,A\right)\right)=b\left(\nu\left(\left(\cdot\right)\cap A\right)\right)=\chi_{A},\quad\forall A\in\mathscr{B}_{fin}.\label{eq:7}
\end{equation}
Extension by linearity and by limits. More generally if $F_{f}\in\mathscr{H}\left(K^{\left(\nu\right)}\right)$
is as in (\ref{eq:n5b}), then $b\left(F_{f}\right)=f\in L^{2}\left(\nu\right)$. 
\end{enumerate}
\end{prop}

\begin{proof}
The details can be found at various places in the literature; see
e.g., \cite{AJ12,zbMATH06664785,2017arXiv170609532J,2017arXiv170708492J}.
\end{proof}
\begin{example}
\label{exa:1d}Below we illustrate the role of kernels $K$, as in
Definition \ref{def:pdk}, in accounting for correlation. The case
of $K_{\left(\nu\right)}$ in (\ref{eq:n4}) yields a RKHS describing
generalized white noise processes; see also (\ref{eq:n5}). By contrast,
if $\left(X,\mathscr{B},\nu\right)$ is given, and we set 
\begin{equation}
K_{2}\left(A,B\right)=\nu\left(A\right)\nu\left(B\right),\quad\forall A,B\in\mathscr{B}_{fin}\label{eq:q1}
\end{equation}
then the corresponding RKHS $\mathscr{H}\left(K_{2}\right)$ is one-dimensional;
and hence trivial. (We sketch the argument below in the special case
when $\nu$ is finite. The reader can easily extend it.) 

To see this, we identify $\left\{ k_{A}\right\} _{A\in\mathscr{B}_{fin}}$
as an $L^{2}\left(\nu\right)$-solution to (\ref{eq:n1}). Let $\mathbbm{1}$
be the constant function ``one'' on $\left(X,\mathscr{B}\right)$.
Assume $\nu\left(X\right)=1$, and set 
\begin{equation}
k_{A}\left(x\right)=\nu\left(A\right)\mathbbm{1}\left(x\right)=\nu\left(A\right),\;\text{then}\label{eq:q2}
\end{equation}
\begin{equation}
\left\langle k_{A},k_{B}\right\rangle _{L^{2}\left(\nu\right)}=\nu\left(A\right)\nu\left(B\right)=K_{2}\left(A,B\right),\label{eq:q3}
\end{equation}
and so, by Corollary \ref{cor:f1}, $\mathscr{H}\left(K\right)$ is
isometrically identified with a one-dimensional subspace of $L^{2}\left(\nu\right)$,
i.e., the subspace spanned by the constant function $\mathbbm{1}$
in $L^{2}\left(\nu\right)$.
\end{example}

\subsection{\label{subsec:TM}Transient Markov Processes}

Below we identify a class of Markov processes which have naturally
associated reproducing kernel Hilbert spaces in the sense of Section
\ref{sec:fact}. Hence, for each of these Markov processes there is
then a positive definite kernel $K$ as per Definitions \ref{def:pdk}
and \ref{def:pdf}. Our starting point is again a fixed sigma-finite
measure space $\left(X,\mathscr{B},\nu\right)$. The Markov transition
probabilities $P$ will be defined relative to $\left(X,\mathscr{B},\nu\right)$.
The two conditions we place on the Markov process is that it be (i)
reversible relative to $\nu$, and that it be (ii) transient; see
definitions below. The reversible property allows us to bring in spectral
theory in $L^{2}\left(\nu\right)$, and the transient property yields
a well defined Greens function which in turn is the key tool in our
consideration of the associated positive definite kernel $K$; and
we show that our results in Section \ref{sec:fact} will then apply.

Let $\left(X,\mathscr{B},\nu\right)$ be as specified above, and let
$P$ denote a Markov transition measure, i.e., 
\begin{equation}
X\ni x\longrightarrow P\left(x,\cdot\right)\in M_{1}\left(X,\mathscr{B}\right)\label{eq:n8}
\end{equation}
assumed measurable, where $M_{1}\left(X,\mathscr{B}\right)=$ probability
measures on $\left(X,\mathscr{B}\right)$. 

We shall assume that the indexed Markov process is reversible (see,
e.g., \cite{2015arXiv150202549J,MR3450534}), i.e., that for all $A,B\in\mathscr{B}_{fin}$,
we have 
\begin{equation}
\int_{A}P\left(x,B\right)d\nu\left(x\right)=\int_{B}P\left(y,A\right)d\nu\left(y\right).\label{eq:n9}
\end{equation}
Set, for $x\in X$, $A\in\mathscr{B}_{fin}$, $n\in\mathbb{N}$, 
\[
P_{n+1}\left(x,A\right)=\int_{X}P\left(x,dy\right)P_{n}\left(y,A\right),
\]
and 
\begin{equation}
G\left(x,A\right)=\sum_{n=0}^{\infty}P_{n}\left(x,A\right).\label{eq:n10}
\end{equation}
We say that the Markov process is transient iff the sum in (\ref{eq:n10})
is pointwise a.e. (w.r.t $\nu$) convergent, and 
\begin{equation}
\int_{A}G\left(x,B\right)d\nu\left(x\right)\label{eq:n11}
\end{equation}
is finite for all $A,B\in\mathscr{B}_{fin}$. (Also see \cite{MR2357627,MR2384480}.)
\begin{prop}
Let $G$ be as in (\ref{eq:n10}). Setting now 
\begin{equation}
K\left(A,B\right)=\int_{A}G\left(x,B\right)d\nu\left(x\right),\label{eq:n12}
\end{equation}
then $K$ is positive definite on $\mathscr{B}_{fin}\times B_{fin}$. 

Moreover, the condition in Corollary \ref{cor:f1} holds. The $\left\{ k_{A}\right\} _{A\in\mathscr{B}_{fin}}$
system satisfying $K\left(A,B\right)=\int_{X}k_{A}\left(x\right)k_{B}\left(x\right)d\nu\left(x\right)$
is 
\begin{equation}
k_{A}=\left(I-P_{\left(\nu\right)}\right)^{-\frac{1}{2}}\left(\chi_{A}\right)\label{eq:n13}
\end{equation}
where $P_{\left(\nu\right)}$ denotes the selfadjoint operator in
$L^{2}\left(\nu\right)$ given by 
\begin{equation}
\left(P_{\left(\nu\right)}\varphi\right)\left(x\right)=\int_{X}\varphi\left(y\right)P\left(x,dy\right)\label{eq:n14}
\end{equation}
\end{prop}

\begin{proof}
With the stated assumptions, the conclusion follows from a direct
application of the results in Section \ref{sec:fact}, i.e., Theorem
\ref{thm:f1}, and Corollary \ref{cor:f1}.

The justification for the formula (\ref{eq:n13}) is as follows: First
$P_{\left(\nu\right)}$ is a bounded operator in $L^{2}\left(\nu\right)$
as follows. It is selfadjoint on account of (\ref{eq:n9}). To show
that it is contractive, consider $\varphi\in L^{2}\left(\nu\right)$,
and estimate: 
\begin{eqnarray*}
\left|\left\langle \varphi,P_{\left(\nu\right)}\varphi\right\rangle _{L^{2}\left(\nu\right)}\right| & = & \left|\iint\varphi\left(x\right)P\left(x,dy\right)\varphi\left(y\right)d\nu\left(x\right)\right|\\
 & \underset{\left(\text{Jensen}\right)}{\leq} & \int\left|\varphi\left(x\right)\right|\left(\int P\left(x,dy\right)\left|\varphi\left(y\right)\right|^{2}\right)^{\frac{1}{2}}d\nu\left(x\right)\\
 & \underset{\left(\text{Schwarz}\right)}{\leq} & \left(\int\left|\varphi\right|^{2}d\nu\right)^{\frac{1}{2}}\left(\iint P\left(x,dy\right)\left|\varphi\left(y\right)\right|^{2}d\nu\left(x\right)\right)^{\frac{1}{2}}\\
 & \underset{\text{by }\left(\ref{eq:n9}\right)}{=} & \int\left|\varphi\right|^{2}d\nu=\left\Vert \varphi\right\Vert _{L^{2}\left(\nu\right)}^{2};
\end{eqnarray*}
and so for the corresponding operator norms, we have: $\left\Vert P_{\left(\nu\right)}\right\Vert _{L^{2}\rightarrow L^{2}}\leq1$,
which is the asserted contractivity. 

One further shows that the transience assumption implies that $\lambda=1$
is not in the point-spectrum of $P_{\left(\nu\right)}$ as an operator
in $L^{2}\left(\nu\right)$. Hence 
\[
\underset{\left(\text{dense subspace}\right)}{\underbrace{\left(I-P_{\left(\nu\right)}\right)L^{2}\left(\nu\right)}}\subsetneq L^{2}\left(\nu\right).
\]
It follows that $\left(I-P_{\left(\nu\right)}\right)^{-1}$ is unbounded
and selfadjoint in $L^{2}\left(\nu\right)$, and so the Spectral Theorem
applies. Hence $\left(I-P_{\left(\nu\right)}\right)^{-\frac{1}{2}}$
is well defined. 

Finally transience implies that (\ref{eq:n11}) is finite. Moreover,
for $A\in\mathscr{B}_{fin}$, we have 
\begin{eqnarray*}
\int_{A}G\left(x,A\right)d\nu\left(x\right) & \underset{=}{\left(\ref{eq:n11}\right)} & \int_{X}\chi_{A}\left(x\right)\left(1-P_{\left(\nu\right)}\right)^{-1}\left(\chi_{A}\right)\left(x\right)d\nu\left(x\right)\\
 & = & \left\Vert \left(1-P_{\left(\nu\right)}\right)^{-\frac{1}{2}}\left(\chi_{A}\right)\right\Vert _{L^{2}\left(\nu\right)}^{2}
\end{eqnarray*}
which justifies the assertion in (\ref{eq:n13}).
\end{proof}

\section{Gaussian Fields}

In this section we establish a close connection between generalized
Gaussian fields, and associated Ito-integrals, on the one hand, and
the class of positive definite kernels considered here, on the other.
This section is motivated in part by a number of earlier works, for
example \cite{AD93,AJ12,MR3616046,MR3624688,MR0277027,MR3507188,MR3642406,MR2529882,MR3504608,MR3531697,MR2357627,MR2558684,MR1201059,MR2384473,MR2851247}.

The setting below will be as in Section \ref{sec:fact}. Fix a sigma-finite
measure space $\left(X,\mathscr{B},\nu\right)$, and a positive definite
kernel $K$ on $\mathscr{B}_{fin}\times\mathscr{B}_{fin}$. We further
assume that $K$ has an $L^{2}\left(\nu\right)$ factorization; see
(\ref{eq:m5}) and (\ref{eq:n1}). In the proof of Theorem \ref{thm:f1},
we saw that there is an unbounded selfadjoint operator $S$, with
dense domain $\mathscr{D}_{fin}$ in $L^{2}\left(\nu\right)$, such
that we may take $\left\{ k_{A}\right\} _{A\in\mathscr{B}_{fin}}$
to have the form $k_{A}=S\left(\chi_{A}\right)$, where $\chi_{A}$
is the indicator function. Hence, we may take as core domain for the
operator $S$, 
\begin{equation}
\mathscr{D}_{fin}:=span\left\{ \chi_{A}\mathrel{;}A\in\mathscr{B}_{fin}\right\} .\label{eq:g1}
\end{equation}
It is a core for the full domain of $S$. Here, by ``span'', we
mean all \emph{finite} linear combinations. (In interesting examples,
see e.g. Section \ref{subsec:TM}, the operator $S$ will indeed be
unbounded. 

\subsection{The path-space $\Omega$}

Let $\dot{\mathbb{R}}$ denote the one-point compactification of $\mathbb{R}$,
and set 
\begin{equation}
\Omega:=\dot{\mathbb{R}}^{\mathscr{B}_{fin}},\label{eq:g2}
\end{equation}
$\mathscr{C}:=$ the cylinder sigma-algebra of subsets of $\Omega$;
and $\mathbb{P}$ be the Gaussian probability measure on $\Omega$,
defined on $\mathscr{C}$, and indexed by $K$. 

In details, if $F=\left\{ A_{i}\right\} _{1}^{n}$ is a finite system,
$A_{i}\in\mathscr{B}_{fin}$, set 
\begin{equation}
\mathscr{C}_{F}:=\prod_{1}^{n}A_{i}\times\prod_{F^{c}}\dot{\mathbb{R}},\label{eq:g3}
\end{equation}
(a cylinder subset); and let $\mathscr{C}$ be the sigma-algebra of
all subsets of $\Omega$ which is generated by the cylinder sets. 

To construct $\mathbb{P}$ as a probability measure, and defined on
$\mathscr{C}$, we first specify its finite-dimensional joint distributions,
\begin{alignat}{2}
\mathbb{P}\left(\cdot\mid\mathscr{C}_{F}\right) & := &  & \text{ the Gaussian on \ensuremath{\mathbb{R}^{n}} which has \ensuremath{0} mean, and }\nonumber \\
 &  &  & \text{ covariance matrix \ensuremath{\left(K\left(A_{i},A_{j}\right)\right)_{i,j=1}^{n}.}}\label{eq:g4}
\end{alignat}
By Kolmogorov's consistency property, we then get a unique probability
measure $\mathbb{P}$ on $\left(\Omega,\mathscr{B}\right)$ which
satisfies (\ref{eq:g4}).

For $A\in\mathscr{B}_{fin}$, and $\omega\in\Omega$ (see (\ref{eq:g2})),
set 
\begin{equation}
W_{A}\left(\omega\right):=\omega\left(A\right).\label{eq:g5}
\end{equation}
Then, by construction, the expectation $\mathbb{E}$, defined from
$\mathbb{P}$, satisfies $\mathbb{E}\left(W_{A}\right)=0$, 
\begin{equation}
\mathbb{E}\left(W_{A}W_{B}\right)=K\left(A,B\right).\label{eq:g6}
\end{equation}
Specifically, each $W_{A}$ is Gaussian with distribution $N\left(0,K\left(A,A\right)\right)$;
and the joint distribution of $\left(W_{A_{1}},W_{A_{2}},\cdots,W_{A_{n}}\right)$
is the Gaussian from (\ref{eq:g4}).

Now, let 
\begin{alignat}{2}
\mathscr{A} & : & = & \;\left\{ A_{i}\right\} =\text{a countable partition of \ensuremath{X}, \ensuremath{A_{i}\in\mathscr{B}_{fin}}}\label{eq:g7}\\
 & \text{} &  & \ensuremath{\;A_{i}\cap A_{j}=\emptyset,}\ensuremath{i\neq j},\nonumber 
\end{alignat}
and let $\mathscr{C}_{\mathscr{A}}:=$ the sigma-subalgebra of $\mathscr{C}$
which is generated by $\left\{ W_{A_{i}}\right\} $, $A_{i}\in\mathscr{A}$.
Introducing conditional expectations, we get an inductive system of
Ito-integrals, indexed by the set of partitions $\mathscr{A}$ (as
above) where we use the usual ordering of partitions $\mathscr{A}\leq\mathscr{A}'$
given by refinement: If $\varphi=\sum_{i}\alpha_{i}\chi_{A_{i}}$,
\begin{equation}
\mathbb{E}\left(\left|\int\varphi dW_{x}\right|^{2}\mid\mathscr{C}_{\mathscr{A}}\right)=\sum\nolimits _{i}\sum\nolimits _{j}\alpha_{i}\alpha_{j}K\left(A_{i},A_{j}\right)\label{eq:g8}
\end{equation}
where the ``$\mid$'' stands for conditional expectation. 

Passing to the limit, over the set of all partitions, we get a necessary
and sufficient condition for when the Ito-integral $\int_{X}\varphi\left(x\right)dW_{x}$
is well defined, and is in $L^{2}\left(\Omega,\mathscr{C},\mathbb{P}\right)$,
i.e., when $\mathbb{E}\left(\left|\int\varphi dW_{x}\right|^{2}\right)<\infty$. 

\subsection{$L^{2}\left(\nu\right)$-factorizations}

We now introduce a particular $L^{2}\left(\nu\right)$ factorization,
\begin{equation}
K\left(A,B\right)=\left\langle k_{A},k_{B}\right\rangle _{L^{2}\left(\nu\right)},\quad A,B\in\mathscr{B}_{fin}\label{eq:g9}
\end{equation}
as in Section \ref{sec:fact}. We introduce the selfadjoint operator
$S$ from (\ref{eq:m13}), so 
\begin{equation}
k_{A}=S\left(\chi_{A}\right).\label{eq:g10}
\end{equation}
For applications, it is not a serious restriction to further assume
that $S$ has a bounded inverse (in $L^{2}\left(\nu\right)$); see
e.g., Section \ref{subsec:TM} above.
\begin{thm}
\label{thm:g1}Let $\left(X,\mathscr{B},\nu\right)$, $\left(\Omega,\mathscr{C},\mathbb{P}\right)$,
and $\left\{ W_{A}\right\} _{A\in\mathscr{B}_{fin}}$ be as above.
We shall consider the corresponding Ito-integral $\int_{X}\varphi\left(x\right)dW_{x}$
for deterministic functions $\varphi$ and $\psi$ on $\left(X,\mathscr{B}\right)$.
We have:
\begin{enumerate}
\item \label{enu:ga1}$\int\varphi dW\in L^{2}\left(\Omega,\mathscr{C},\mathbb{P}\right)$
iff $S\left(\varphi\right)\in L^{2}\left(\nu\right)$, i.e., iff $\varphi$
is in the domain of $S$. 
\item When (\ref{enu:ga1}) holds, then 
\begin{equation}
\mathbb{E}\left(\left|\int\varphi dW\right|^{2}\right)=\left\Vert S\left(\varphi\right)\right\Vert _{L^{2}\left(\nu\right)}^{2},\label{eq:g11}
\end{equation}
and 
\begin{equation}
\mathbb{E}\left(\left(\int\varphi dW\right)\left(\int\psi dW\right)\right)=\left\langle S\varphi,S\psi\right\rangle _{L^{2}\left(\nu\right)}.\label{eq:g12}
\end{equation}
\end{enumerate}
\end{thm}

\begin{proof}
In view of the discussion before the statement of the theorem, it
is enough to prove (\ref{eq:g12}). We shall do this by first establishing
the formula when a partition is fixed. The general case will then
follow from the inductive limit argument outlined above. 

In detail: Since the Ito-integral is defined as an inductive limit
on partitions as specified in (\ref{eq:g8}), to prove (\ref{eq:g12}),
it is enough to fix a partition $\mathscr{A}=\left\{ A_{i}\right\} $,
and functions $\varphi=\sum_{i}\alpha_{i}\chi_{A_{i}}$, $\psi=\sum_{j}\beta_{j}\chi_{A_{j}}$;
then
\begin{eqnarray*}
\mathbb{E}\left(\left(\int\varphi dW\right)\left(\int\psi dW\right)\right) & = & \sum\nolimits _{i}\sum\nolimits _{j}\alpha_{i}\beta_{j}\mathbb{E}\left(W_{A_{i}}W_{A_{j}}\right)\\
 & \underset{\left(\text{by \ensuremath{\left(\ref{eq:g6}\right)}}\right)}{=} & \sum\nolimits _{i}\sum\nolimits _{j}\alpha_{i}\beta_{j}K\left(A_{i},A_{j}\right)\\
 & \underset{\left(\text{by \ensuremath{\left(\ref{eq:g10}\right)}}\right)}{=} & \sum\nolimits _{i}\sum\nolimits _{j}\alpha_{i}\beta_{j}\left\langle S\left(\chi_{A_{i}}\right),S\left(\chi_{A_{j}}\right)\right\rangle _{L^{2}\left(\nu\right)}\\
 & = & \left\langle S\left(\sum\nolimits _{i}\alpha_{i}\chi_{A_{i}}\right),S\left(\sum\nolimits _{j}\beta_{j}\chi_{A_{j}}\right)\right\rangle _{L^{2}\left(\nu\right)}\\
 & = & \left\langle S\varphi,S\psi\right\rangle _{L^{2}\left(\nu\right)}
\end{eqnarray*}
which is the desired conclusion (\ref{eq:g12}) when the partition
$\mathscr{A}$ is fixed. 

The general result follows from use of the ordering of partitions,
and the corresponding Kolmogorov inductive limit over all partitions;
see (\ref{eq:g7}). For more details, we refer to the cited literature
on standard considerations for generalized Ito limits.
\end{proof}
Let the setting be as in the statement of Theorem \ref{thm:g1}. We
consider the limit over partitions $\mathscr{A}$ as in (\ref{eq:g7}):
\begin{equation}
\mathbb{E}\left(\left|\int\varphi dW\right|^{2}\mid\mathscr{C}_{\mathscr{A}}\right),\label{eq:h1}
\end{equation}
where the limit is taken over $\mathscr{A}$. 

Here $\mathscr{C}_{\mathscr{A}}:=$ (the sigma-algebra generated by
$\left\{ W_{A_{i}}\right\} $ for $A_{i}\in\mathscr{A}$); and we
take monotone limit over refinements, written $\mathscr{A}\leq\mathscr{A}'$.
Note that 
\begin{equation}
\mathbb{E}\left(\left|\int\varphi dW\right|^{2}\mid\mathscr{C}_{\mathscr{A}}\right)\leq\mathbb{E}\left(\left|\int\varphi dW\right|^{2}\mid\mathscr{C}_{\mathscr{A}'}\right)\label{eq:h2}
\end{equation}
holds if $\mathscr{A}\leq\mathscr{A}'$. 
\begin{cor}[Monotonicity]
 We have: 
\begin{equation}
\sup_{\mathscr{A}}\mathbb{E}\left(\left|\int\varphi dW\right|^{2}\mid\mathscr{C}_{\mathscr{A}}\right)=\mathbb{E}\left(\left|\int\varphi dW\right|^{2}\right)\label{eq:h3}
\end{equation}
where the supreme in (\ref{eq:h3}) over all partitions is finite
iff $\varphi\in domain\left(S\right)$. 
\end{cor}

\begin{proof}
From the construction, we get 
\[
\mathbb{E}\left(\cdot\cdot\mid\mathscr{C}_{\mathscr{A}}\right)\leq\mathbb{E}\left(\cdot\cdot\mid\mathscr{C}_{\mathscr{A}'}\right)
\]
in the ordering of projections in $L^{2}\left(\Omega,\mathscr{C},\mathbb{P}\right)$.
Equivalently, when $\mathscr{A}\leq\mathscr{A}'$, we have:
\[
\mathbb{E}\left(\left(F\mid\mathscr{C}_{\mathscr{A}'}\right)\mid\mathscr{C}_{\mathscr{A}}\right)=\mathbb{E}\left(F\mid\mathscr{C}_{\mathscr{A}}\right),
\]
and the assertions in (\ref{eq:h2})-(\ref{eq:h3}) follow immediately
from this.
\end{proof}
\begin{acknowledgement*}
The co-authors thank colleagues for helpful and enlightening discussions,
and members in the Math Physics seminar at The University of Iowa.
\end{acknowledgement*}
\bibliographystyle{amsalpha}
\bibliography{ref}

\end{document}